\documentclass[11pt,a4paper]{article}

\usepackage[ngerman,USenglish]{babel}
\usepackage{amssymb}
\usepackage{amsfonts}
\usepackage{amsmath}
\usepackage{amsthm}
\usepackage{graphics}
\usepackage[usenames, dvipsnames]{xcolor}
\usepackage[hidelinks]{hyperref} 
\usepackage{dsfont}
\usepackage{verbatim}
\usepackage{xcolor}
\definecolor{myblue}{cmyk}{0.5, 0.1, 0.1, 0.1}
\usepackage[textsize=tiny,backgroundcolor=myblue,linecolor=myblue,]{todonotes}
\usepackage{csquotes}
\usepackage{enumitem} 

\allowdisplaybreaks[2]

\numberwithin{equation}{section}

\newtheorem{theorem}{Theorem}[section]
\newtheorem{lemma}[theorem]{Lemma}
\newtheorem{proposition}[theorem]{Proposition}
\newtheorem{corollary}[theorem]{Corollary}

\theoremstyle{definition}
\newtheorem{definition}[theorem]{Definition}
\newtheorem{example}[theorem]{Example}
\theoremstyle{remark}
\newtheorem{remark}[theorem]{Remark}

\newcommand*\diff{\mathop{}\!\mathrm{d}}

\newcommand{\bR}{\mathbb{R}}

\newcommand{\cE}{\mathcal{E}}

\newcommand{\EE}{\mathbb{E}}

\newcommand{\PP}{\mathbb{P}}

\newcommand{\cF}{\mathcal{F}}

\newcommand{\NN}{\mathbb{N}}

\newcommand{\RR}{\mathbb{R}}
\newcommand{\FF}{\mathbb{F}}

\newcommand{\one}{\mathds{1}}

\newcommand{\di}{\mathrm{d}}
\newcommand{\re}{\mathrm{e}}
\newcommand{\ri}{\mathrm{i}}

\usepackage{marginnote}
\newcommand\normal{\color{black}}

\newcommand{\AL}{\color{black}}
\newcommand\HeHe{\color{black}}

\begin{document}
	\author{Anita Behme\thanks{Technische Universit\"at Dresden,
			Institut f\"ur Mathematische Stochastik, 01062 Dresden, Germany, e-mail: anita.behme@tu-dresden.de, henriette.heinrich@tu-dresden.de}$\,$, Henriette E. Heinrich$^\ast$, and Alexander Lindner\thanks{Universit\"at Ulm, Institut f\"ur Finanzmathematik, 89081 Ulm, Germany, e-mail: alexander.lindner@uni-ulm.de} }
	\title{Duals and inverse flows of generalized Ornstein-Uhlenbeck processes}
	\maketitle
	
	\begin{abstract}
	We derive explicit representations for the (Siegmund) dual and the inverse  flow of generalized Ornstein-Uhlenbeck processes whenever these exist.
	It turns out that the dual and the process corresponding to the inverse stochastic flow are again generalized Ornstein-Uhlenbeck processes. Further, we observe that the stationary distribution of the dual process provides information about the hitting time of zero of the original process.
	\end{abstract}
	
	
	

	\section{Introduction}

\AL Half a century ago Siegmund \cite[Thm. 1]{Siegmund1976} characterized when a  $[0,\infty]$-valued Markov process has a  so-called \emph{dual} $[0,\infty]$-valued Markov process. \HeHe According to his results, under mild regularity assumptions, the relevant conditions on the Markov process are stochastic monotonicity and right-continuity of the associated tail function in the starting value. \AL In this paper, we shall not work with $[0,\infty]$-valued Markov processes but with generalized Ornstein--Uhlenbeck processes, which are particular $\bR$-valued time-homogeneous Markov processes.
In analogy to Siegmund's original definition,
given an  $\RR$-valued  (universal) time-homogeneous  Markov process $(X_t^x)_{t\geq 0}$ (where $X_0^x=x\in\RR$ denotes the starting value),
another  $\RR$-valued  (universal) time-homogeneous Markov process $(Y_t^y)_{t\geq 0}$ (with starting value  $Y_0^y = y\in\RR$) will be called \emph{(Siegmund) dual} to $(X_t^x)_{t\geq 0}$, if
\begin{equation}\label{dual}
	\PP(X_t^x\geq y) = \PP(Y_t^y\leq x), \quad  0\leq t<\infty, \; x,y\in \RR,
\end{equation} i.e.
$\PP(X_t \geq y|X_0=x) = \PP(Y_t \leq x|Y_0=y)$.
\normal


	Note that, similar to the notation in \cite[Def. 7.1.1 f.]{OksendalSDE2003} or \cite[Def. 1.21]{Oksendal2007}, in the following and as above, we denote by $(X_t^x)_{t \ge 0}$ the time-homogeneous universal Markov process $(X_t)_{t \ge 0}$ started in $x$, i.e. $\PP(X_t^x \in B) = \PP_x(X_t \in B) = \PP(X_t \in B |X_0 = x)$ for all $B \in \mathcal{B}(\bR), t \ge 0$. \AL Also note that the definition of Siegmund duality in \eqref{dual} is a special case of more general dualities of Markov processes; see \cite{Jansen2014} for an overview.\normal
	
The duality relation \eqref{dual} is widely known and commonly used in applied probability, as it provides a powerful tool to link hitting probabilities and asymptotic behavior of dual processes.	 To illustrate this, let $(Y_t^y)_{t\geq 0}$ be dual to $(X_t^x)_{t\geq 0}$ and consider $\tau(y):=\inf\{t \ge 0 \colon Y_t^y \le 0\}$
and assume that $(Y_t^y)_{t\geq 0}$ is \AL absorbed when it first becomes non-positive. \normal
Then, under suitable regularity conditions,
	\begin{equation*}
		\PP(\tau(y)<\infty) = \lim_{t\to \infty} \PP(Y_t^y \leq 0) \overset{\eqref{dual}}{=}  \lim_{t\to \infty}\PP(X_t^0\geq y ) = \PP(X\geq y) \quad \AL \forall\; y \geq 0,\normal
	\end{equation*}
	where $X$ is a random variable whose law is given by the limiting distribution of $(X_t^0)_{t\ge0}$, provided it exists. 

This relationship has been used to determine entrance and exit laws or behaviour at reflecting barriers in the context of birth-death chains, see e.g. \cite{CoxRoesler1984, Dette1997, DiaconisFill1990, Huillet2010, HuilletMartinez2011}, in actuarial mathematics and mathematical finance  \cite{Asmussen95, AsmussenAlbrecher, AsmussenSigman1996, Goffard2019, Kolokoltsov2015}, as well as in game theory \cite{Lorek2017, Lorek2018} and population genetics \cite{Cordero2025, Foucart2019, FoucartVidmar2024, KuklaMoehle2018, Moehle2018}, just to mention a few. In the setting of interacting particle systems Siegmund duality has been exhaustively studied, see e.g. \cite{Liggett} for a textbook reference.
An analysis of Siegmund duality for multidimensional processes was carried out in \cite{BlaszczyszynSigman1999, KolokoltsovLee2013}, where the first reference  discusses set-valued dual processes and the latter uses a generator approach for Siegmund duality based on Pareto ordering. A recent reference for applications of \eqref{dual} in physics is \cite{Gueneau2024}.


In this paper, we treat Siegmund duality for the \HeHe above mentioned \normal 
setting of  (time-homogeneous) Markov processes with state space~$\bR$. 
In particular, we consider solutions of the stochastic differential equation (SDE)
\begin{equation}
	\diff V_t^x = V_{t-}^x \, \diff U_t + \diff L_t, \quad V_0^x = x, \label{eq:GOUSDE}
\end{equation}
for a general bivariate driving Lévy process $(U_t,L_t)_{t\geq 0}$. Assuming that $U$ does not admit jumps of size $-1$, these solutions are known as \emph{generalized Ornstein-Uhlenbeck (GOU) processes}  and they can be expressed explicitly; see \cite[Thm. 2.1]{BLM2011}. We use this fact to prove in Section~\ref{S-dual}\HeHe, Theorem~\ref{thm:dualGOU}, \normal that the dual of a GOU process $(V_t^x)_{t\geq 0}$, if existent, is again a GOU process, namely a process $(R_t^y)_{t\geq 0}$ solving an SDE of the form \eqref{eq:GOUSDE}. We furthermore discuss the relation between hitting probabilities and stationary distributions of the dual pair of GOU processes. As it turns out, \HeHe see Corollary~\ref{c-dual-stationary} below, \normal the appearing (causal) stationary distribution of the dual GOU process coincides with the non-causal stationary distribution of the original GOU process. When $L$ is a subordinator \AL and $U$ only has jumps of size $>-1$, \normal then $(V_t^x)_{t\geq 0}$ can also be viewed as a $[0,\infty)$-valued Markov process when only $x\geq 0$ is allowed, while the derived dual GOU process $(R_t^y)_{t\geq 0}$ \emph{on all of $\bR$} is driven by the negative of a subordinator, hence becomes also negative, even when $y\geq 0$. Hence, the Siegmund dual process $(\widehat{R}_t^y)_{t \ge 0}$ of $(V_t^x)_{t\geq 0}$, $x\geq 0$, viewed as an $[0,\infty]$-valued process \AL (i.e.~as in the original setting of Siegmund~\cite{Siegmund1976}), \normal must be different from the Siegmund dual GOU process $(R_t^y)_{t\geq 0}$, viewed as a Markov process on $\bR$ \AL (i.e. as defined above). \normal
In Proposition~\ref{prop-dualhalfline} we clarify the relation between $R$ and $\widehat{R}$.

Note that despite the fact that GOU processes are known to be Feller processes as long as $U$ has no jumps of size $-1$, see \cite{BLexpfunc, kuehn}, the approach via the dual generator provided in \cite{Kolokoltsov2011} is not fully applicable in the present setting. This is partly due to the fact that in \cite{Kolokoltsov2011}, the analysis is restricted to processes with bounded coefficients (although it is stated that this can be dropped), which does not hold for GOU processes. Moreover, the assumptions made in \cite{Kolokoltsov2011} to guarantee stochastic monotonicity and hence existence of a dual process would imply that only positive jumps of the L\'evy process $U$ are admissible. This however is not necessary as we shall see in Lemma~\ref{lem-GOUstochmon} below.

	As it is often the case for dual processes, there is an intrinsic relation between time-reversal of a Markov process (more precisely, the stochastic process induced by the inverse stochastic flow) and its dual process; see e.g. ~\cite{DiaconisFill1990, Nagasawa1964}. In case of processes with continuous paths, this connection follows easily, as the time-reversed process shifted along the vertical axis provides a version of the Siegmund dual process; see e.g. \cite{Sigman2000}. However, path continuity is not a necessary condition, as, e.g.  the duality relation of the Cram\'er-Lundberg risk process absorbed in ruin and the M/G/1-queue workload process described in Example \ref{ex:risk} below can be proved via a time-reversion argument; see \cite[Chapter III.2]{AsmussenAlbrecher}.
Motivated by this, we turn our attention to the inverse flow of GOU processes in Section \ref{S-timereverse} and prove \HeHe in Theorem~\ref{thm:time_reversed_GOU} \normal  that the process associated with the inverse stochastic flow is again a GOU process and, moreover, that it provides a version of the dual GOU process.

This study of time inverted GOU processes is very timely and of particular interest by itself, as time-reversals of OU processes are fundamental for the most prominent denoising diffusion models in image generation; see e.g. \cite{Croitoru2023}. Recent studies, such as \cite{YPKL}, outline the potential of (time inverted) jump diffusions in this context, and GOU processes are a natural candidate to pursue this approach.

	\section{Preliminaries}
	\setcounter{equation}{0}
	
	General information regarding L\'evy processes can be found e.g. in \cite{Bertoin1996, BrockwellLindner2024, Sato1999}, and for stochastic integration we refer to \cite{Protter2003}.
	
	\subsection{Generalized Ornstein-Uhlenbeck processes}\label{SPrelimGOU}

	Throughout, let $(U_t,L_t)_{t\geq 0}$ be a bivariate L\'evy process with characteristic triplet $(\gamma_{U,L}, \Sigma_{U,L}, \nu_{U,L})$ for some location parameter $\gamma_{U,L}\in \bR^2$, a non-negative definite Gaussian covariance matrix $\Sigma_{U,L}=\left(\begin{smallmatrix}
		\sigma_U^2 & \sigma_{U,L} \\
		\sigma_{U,L} & \sigma_L^2
	\end{smallmatrix} \right)\in \bR^{2\times 2}$, and a Lévy measure $\nu_{U,L}$ on $\RR^2\setminus \{0\}$.
	Then the characteristic function of $(U_t,L_t)_{t\geq 0}$ at $(x,y)^T\in \bR^2$ is given by
	\begin{align*}
		\mathbb{E} \re^{\ri (x U_t + y L_t)} = & \exp \Big\{  t \Big( \ri (x,y) \gamma_{U,L}^T - (x,y) \Sigma_{U,L} (x,y)^T /2\\
		& \qquad + \int_{\bR^2 \setminus \{0\}} (\re^{\ri (x,y) z^T} - 1 - \ri (x,y) z^T \one_{\{|z|\leq 1\}}) \, \nu_{U,L}(\diff z) \Big) \Big\},
	\end{align*}
	where the superscript $T$ denotes transposition. The marginal location parameters and the marginal jump measures of $U=(U_t)_{t\geq 0}$ and $L=(L_t)_{t\geq 0}$ are denoted by $\gamma_U$, $\gamma_L$, $\nu_U$, and $\nu_L$, respectively.
	
	For any c\`adl\`ag process $Z$ we denote by $Z_{t-}$ the left-hand limit of $Z$ at time $t\in(0,\infty)$, set $Z_{0-}:=Z_0$, and write $\Delta Z_t=Z_t-Z_{t-}$ for its jumps. Throughout, we make the assumption
	\begin{equation}
		\nu_{U,L}(\{-1\}\times \RR)=0, \text{ i.e. } \Delta U_t\neq -1 \text{ for all } t\geq 0, \label{eq:assU(A)}
	\end{equation}
	when considering the SDE \eqref{eq:GOUSDE}, as it implies \AL that the unique solution is of the form as stated in Lemma \ref{lemma:BLMProp3.2} below. \normal

\HeHe
	Recall that the stochastic exponential $(\cE(U)_t)_{t\geq 0}$ of $(U_t)_{t\geq 0}$ is the unique solution of the SDE
	$$\diff \cE(U)_t = \cE(U)_{t-} \diff  U_t ,\quad \cE(U)_0=1,$$
	which can be expressed explicitly by the Dol\'eans-Dade formula (see~\cite[Thm.~II.37]{Protter2003})  as
	\begin{equation}\label{Dolean}
		\cE(U)_t = \re^{U_t - \sigma_U^2 t/2} \prod_{0<s\leq t} (1+\Delta U_s) \re^{-\Delta U_s}, \quad t \geq 0.
	\end{equation}
	From this it follows that $\cE(U)_t \neq 0$ for all $t$ under Assumption \eqref{eq:assU(A)}, while $\cE(U)_t > 0$ for all $t$ if and only if $\Delta U_t > -1$ for all $t$. \normal
	
	\begin{lemma}[{\cite[Prop. 3.2]{BLM2011}}]\label{lemma:BLMProp3.2}
		Under Assumption \eqref{eq:assU(A)}, the solution $(V_t^x)_{t\geq 0}$ of \eqref{eq:GOUSDE}, called the \emph{generalized Ornstein-Uhlenbeck (GOU) process}, is unique and given explicitly by
		\begin{align}\label{eq:GOU-explicit}
			V_t^x = \cE(U)_t \Big( x+ \int_{(0,t]} \cE(U)_{s-}^{-1} \diff \eta _s\Big), \quad t\geq 0.
		\end{align}
Hereby, $(\cE(U)_t)_{t\geq 0}$ is the stochastic exponential of $(U_t)_{t\geq 0}$ and $\eta=(\eta_t)_{t\geq 0}$ is the L\'evy process given by
\begin{equation} \label{eq:etaviaUL}
		\eta_t =L_t - \sum_{0<s\leq t} \frac{\Delta U_s \Delta L_s}{1+ \Delta U_s} - t \sigma_{U,L}, \quad t\geq 0.
	\end{equation}
	\end{lemma}

	Observe  that if $U$ and $L$ are independent, then $\Delta U_t \Delta L_t = 0$ for any $t \ge 0$ and $\sigma_{U,L} = 0$, hence $\eta = L$ by \eqref{eq:etaviaUL}. This is particularly the case if $U_t=-\lambda t$ is chosen deterministically and we refer to the resulting GOU process
		\begin{align*}
			V_t^x = \re^{-\lambda t} \Big( x+ \int_{(0,t]} \re^{\lambda s} \diff \eta _s\Big)
		\end{align*}
		as \emph{L\'evy-driven Ornstein-Uhlenbeck process}.

If the Assumption \eqref{eq:assU(A)} is strengthened to
\begin{equation} \label{eq:assU(B)}
\nu_{U,L}((-\infty,-1] \times \RR)=0, \text{ i.e. } \Delta U_t> -1 \text{ for all } t\geq 0,
\end{equation}
then $\cE(U)_t > 0$ for all $t$ and one can define another Lévy process  $\xi=(\xi_t)_{t \ge 0}$ by $\xi_t := -\log \cE(U)_t$ for all $t \ge 0$, resulting in the more common explicit representation of a GOU process as
	\begin{align}\label{eq:GOU-explicitxi}
		V_t^x = \re^{-\xi_t} \Big( x+ \int_{(0,t]} \re^{\xi_{s-}} \diff \eta _s\Big), \quad t\geq 0.
	\end{align}
	If additionally $L$ is a subordinator, then also $\eta$ is a subordinator, since in this case $\sigma_{U,L} = 0$ and $\Delta \eta_t = \Delta L_t / (1+\Delta U_t)$. Further, due to \eqref{eq:GOU-explicitxi}, in this case the resulting GOU process started in some $x>0$ only takes values in $(0,\infty)$.\\
	Given the bivariate Lévy process $(\xi,\eta)$, the process $(U,L)$ can be recovered via
	\begin{equation}\label{eq:BLM(1.3)}
		\begin{pmatrix}
			U_t \\ L_t
		\end{pmatrix} = \begin{pmatrix}
			- \xi_t + \sum_{0<s\leq t} (\re^{-\Delta \xi_s} -1 + \Delta \xi_s) + t \sigma_{\xi}^2/2 \\
			\eta_t + \sum_{0<s\leq t} (\re^{-\Delta \xi_s} -1)\Delta \eta_s - t \sigma_{\xi,\eta}
		\end{pmatrix}, \quad t \ge 0,
	\end{equation}
	see \cite[Eq. (1.3)]{BLM2011}.\\
	
	Throughout the paper, let $\FF=(\cF_t)_{t\geq 0}$ denote the augmented natural filtration of $(U_t,L_t)_{t\geq 0}$, and note that by \cite[Thm. 3.1]{BLexpfunc} (see also \cite{kuehn}), $(V_t^x)_{t\geq 0}$ is a rich Feller process and hence a Markov process with respect to $\FF$.\\
	
	Our exposition will further rely on the following lemmata.
	\begin{lemma}[{\cite[Lemma 3.1]{BLM2011}}]\label{lemma:BLMLemma3.1}
		Suppose  \eqref{eq:assU(A)} is satisfied, then for every $t \ge 0$,
		\begin{equation*}
			\begin{pmatrix}
				\cE(U)_t \\
				\cE(U)_t \int_{(0,t]} \cE(U)_{s-}^{-1} \diff \eta_s
			\end{pmatrix} \overset{d}{=} \begin{pmatrix}
				\cE(U)_t \\
				\int_{(0,t]} \cE(U)_{s-} \diff L_s
			\end{pmatrix}.
		\end{equation*}
	\end{lemma}
	
	\begin{lemma}[{\cite[Lemma 3.4]{BLM2011}} or {\cite[Thm. 1]{Karandikar1991}}]\label{lemma:BLMLemma3.4}
		Assume  \eqref{eq:assU(A)}  and
		let $W=(W_t)_{t\geq 0}$ be defined by
		\begin{equation*}
			W_t := -U_t + \sigma_U^2 t + \sum_{0<s\leq t} \frac{(\Delta U_s)^2}{1+ \Delta U_s}, \quad t\geq 0.
		\end{equation*}
		Then $W$ is a L\'evy process satisfying
		\begin{equation*}
			\cE(U)_t^{-1} = \cE(W)_t, \quad t \ge 0,
		\end{equation*}
		and its characteristic triplet $(\gamma_W, \sigma_W^2, \nu_W)$ is given by
		\begin{equation*}
			\gamma_W = - \gamma_U + \sigma_U^2 + \int_{\bR} \left(z \mathds{1}_{\{|z|\le 1\}} - \frac{z}{1+z} \mathds{1}_{\{z \ge -1/2\}}\right) \nu_U(\diff z),
		\end{equation*}
		 $\sigma_W^2 = \sigma_U^2$, and $\nu_W = F(\nu_U)$ (the image measure) for
		\begin{equation*}
			F \colon \bR \setminus \{-1\} \to \bR \setminus \{-1\}, \quad x \mapsto \frac{-x}{1+x}.
		\end{equation*}		
	\end{lemma}

Consider \eqref{eq:GOU-explicit} with an arbitrary, not necessarily deterministic starting random variable $V_0$, i.e.
\begin{equation} \label{eq-random-start}
V_t^{V_0} = \cE(U)_t \Big( V_0 + \int_{(0,t]} \cE(U)_{s-}^{-1} \diff \eta _s\Big), \quad t\geq 0.
\end{equation}
In \cite{BLM2011} it was examined under which conditions $V_0$ can be chosen such that $(V_t^{V_0})_{t\geq 0}$ is stationary. Here, $(V_t^{V_0})_{t\geq 0}$ is called \emph{stationary} if its finite-dimensional distributions are shift-invariant. The distribution of $V_0$ is then called a \emph{stationary distribution of $(V_t^x)_{t\geq 0}$}. The stationary distribution is clearly an invariant distribution of the Markov process $(V_t^x)_{t\geq 0}$.

	\begin{lemma}[{\cite[Thm. 2.1]{BLM2011}} and its proof] \label{lemma:BLMThm2.1}
Assume \eqref{eq:assU(A)}.
\begin{enumerate}[label=(\alph*)]
\item
Suppose that $\lim_{t\to \infty} \cE(U)_t=0$ almost surely.
Then a finite random variable $V_0$ can be chosen such that $(V_t^{V_0})_{t\geq 0}$ defined by
\eqref{eq-random-start} is stationary if and only if the limit
		\begin{equation}
			\int_{\HeHe (0,\infty)} \cE(U)_{s-} \diff L_s := \lim_{t\to\infty} \int_{(0,t]} \cE(U)_{s-} \diff L_s \label{eq:thmcausalsol}
		\end{equation}
		exists almost surely and is finite. In that case, the stationary distribution is unique, given by  the law of \eqref{eq:thmcausalsol}, and it is the limit distribution
of $(V_t^x)_{t\geq 0}$ as $t\to\infty$ for arbitrary $x\in \bR$. The stationary process $(V_t^{V_0})_{t\geq 0}$ in this case is also unique in distribution and the starting random variable $V_0$ is independent of $(U_t,L_t)_{t \geq 0}$.
\item Suppose that $\lim_{t\to \infty} \cE(U)_t^{-1} =0$ almost surely. Then a finite random variable $V_0$ can be chosen such that $(V_t^{V_0})_{t\geq 0}$ defined by \eqref{eq-random-start} is stationary if and only if the limit
		\begin{equation}
			-\int_{\HeHe (0,\infty)} \cE(U)_{s-}^{-1} \, \di \eta_s := - \lim_{t\to \infty} \int_{(0,t]} \cE(U)_{s-}^{-1} \, \diff \eta_s, \label{eq:thmnoncausalsol}
		\end{equation}
		exists almost surely and is finite. In that case, the stationary solution is unique and given by $V_t^{V_0} = - \cE(U)_t \int_{(t,\infty)} (\cE(U)_t)^{-1} \, \diff \eta_s$. In particular $V_0$ (and hence the stationary distribution) is unique and given by \eqref{eq:thmnoncausalsol}.
	\end{enumerate}
\end{lemma}

The stationary solution $(V_t^{V_0})_{t\geq 0}$ obtained in (a) is called \emph{causal}, since there $V_0$ is independent of $(U_t,L_t)_{t\geq 0}$ (equivalently, of $(U_t,\eta_t)_{t\geq 0}$), and the distribution of $V_0$ is then called a \emph{causal stationary distribution}. On the other hand, $V_0$ obtained in case (b) is not independent of $(U,L)$ (unless $V_0$ degenerates to a constant), which is why we refer to the solution $(V_t^{V_0})_{t\geq 0}$ and the distribution of $V_0$ in case (b) as \emph{non-causal stationary solution/distribution}. Observe that the stationary distribution is the limit distribution of $(V_t^x)_{t\geq 0}$ for arbitrary $x\in \bR$ in the causal case, while this is in general not true in the non-causal case.

	Necessary and sufficient conditions for almost sure convergence of the \HeHe so-called \emph{exponential functionals} \normal $\int_{(0,\infty)} \cE(U)_{s-} \diff L_s$ and $-\int_{\HeHe (0,\infty)} \cE(U)_{s-}^{-1} \, \di \eta_s$ have been obtained by Erickson and Maller~\cite{EricksonMaller2005}, conditions for the almost sure convergence of $\cE(U)_t$ or $\cE(U_t)^{-1}$ to 0 in Doney and Maller~\cite{DoneyMaller2002}, and all these results are summarised in \cite[Thm.~3.5 -- Cor. 3.7]{BLM2011}. Further, as shown in \cite[Thm.~2.1]{BLM2011}, if neither $\cE(U)_t$ nor $\cE(U)_t^{-1}$ converges almost surely to $0$ as $t\to\infty$, and neither $U$ nor $L$ is the zero-process, then choices of $V_0$ making $(V_t^{V_0})_{t\geq 0}$ stationary exist only in degenerate cases, namely if there exists $k\in\RR\setminus\{0\}$ such that almost surely
	\begin{equation}\label{eq-degenerate} k\cdot U_t=-L_t, \quad t\geq 0.\end{equation}
As shown in \cite[proof of Thm.~2.1]{BLM2011}, Equation~\eqref{eq-degenerate} for $k\neq 0$ is  equivalent to
	\begin{equation}\label{eq-degenerate2}
\int_{(0,t]} \cE(U)_{s-} \diff L_s=k(1-\cE(U)_t) \quad \text{for each $t\geq 0$},
\end{equation}
and the strictly stationary solution is then indistinguishable from the constant process $V_t\equiv k$. \HeHe Distributional properties of exponential functionals, i.e. of the stationary distributions of GOU processes, are widely studied, see e.g. \cite{CarmonaPetitYor,Paulsen,Yorbook} for early studies in this field providing concrete examples under specific assumptions on the driving Lévy processes, \cite{Behme2015, BLexpfunc} for results in a broader context, or the recent review \cite{MinchevSavov}. \normal
\AL For example, when $U$ is a standard Brownian motion, then $\cE(U)_t = \re^{U_t -  t/2}$ converges almost surely to 0 and it follows from the results of Erickson and Maller~\cite{EricksonMaller2005} that $\int_{(0,\infty)} \cE(U)_{s-} \diff L_s$ converges almost surely to a finite random variable if and only if $|L_1|$ has finite log-moment.\normal

	Lastly, as this is of interest in the study of duality, we add the following lemma concerning stochastic monotonicity of GOU processes.
	
	\begin{lemma}\label{lem-GOUstochmon}
Suppose that  \eqref{eq:assU(A)} is satisfied. Then
		the GOU process $(V_t^x)_{t\geq 0}$ solving \eqref{eq:GOUSDE} is stochastically monotone (in the sense that the mapping $x \mapsto \PP (V_t^x \geq y)$ is non-decreasing in $x\in\RR$ for any $y\in \bR$ and $t\geq 0$) if and only if \eqref{eq:assU(B)} is satisfied.
	\end{lemma}
	\begin{proof}
		\AL First \normal assume \eqref{eq:assU(B)}, i.e. $\Delta U_t>-1$ for all $t$. By \eqref{Dolean} we have $\cE(U)_t >0$ for all $t \ge 0$ and hence  for fixed $t$
		the mapping $$x \mapsto V_t^x = \cE(U)_t \Big( x+ \int_{(0,t]} \cE(U)_{s-}^{-1} \diff \eta _s\Big) $$
		is an increasing, affine function in $x$. Thus, for $x_1 \le x_2$ and $y \in \bR$ we have $\{V_t^{x_1} \ge y\} \subset \{V_t^{x_2} \ge y\}$, and therefore
		$$ \PP(V_t^{x_1} \ge y) \le \PP(V_t^{x_2} \ge y),$$
		thus proving stochastic monotonicity. \\
		\AL For the converse implication, assume that $U$ has jumps of size $<-1$ with positive probability, \normal
		then $\PP(\cE(U)_t < 0) > 0$ by \eqref{Dolean}. Now if $(V_t^x)_{t\geq 0}$ were stochastically monotone, then for all fixed $x,y\in \bR$ and $t>0$ we would have
		\begin{align*}
			\PP(V_t^x \geq y) &\leq \lim_{z\to\infty} \PP(V_t^z \geq y) \\
			& = \lim_{z\to\infty}
			\PP\bigg( \cE(U)_t \left( z +    \int_{(0,t]} \cE(U)_{s-}^{-1} \diff \eta _s \right) \geq y, \cE(U)_t>0\bigg) \\
				& \quad + \lim_{z\to\infty} \PP\bigg( \cE(U)_t \left( z +    \int_{(0,t]} \cE(U)_{s-}^{-1} \diff \eta _s \right) \geq y, \cE(U)_t<0\bigg)  \\
			& = \PP(\cE(U)_t > 0) .
		\end{align*}
		In particular
		$$1 = \lim_{y\to -\infty} \PP(V_t^x \geq y) \leq \PP(\cE(U)_t > 0) < 1,$$
		for $x\in \bR$ and $t>0$, which is a contradiction.		
		Hence $(V_t^x)_{t \ge 0}$ cannot be stochastically monotone.
	\end{proof}

	\subsection{Hitting probabilities}
	
	As described in the introduction, Siegmund's duality is often used to match ruin probabilities, i.e. hitting probabilities of the negative half line, with stationary distributions of the dual processes. To further illustrate this connection, consider the following example which falls in the setting of GOU processes and motivates our considerations in Corollary \ref{cor-GOUhitting}.
	
	\begin{example}\label{ex:risk}
		Assume that $ (Y_t^y)_{t\geq 0}= (R_t^y)_{t\geq 0}$ is a Cramér-Lundberg risk process absorbed in ruin. This process can be defined by
		$$R_t^y := (y+K_t) \, \one_{\{t < \tau(y)\}}, \quad t,y \geq 0,$$
		where
		\begin{align*}
			K_t & := ct - \sum_{i=1}^{N_t} S_i, \qquad
			\tau(y)  := \inf \{ t\geq 0 : y+ K_t \leq 0\},
		\end{align*}
		for some premium rate $c\geq 0$, a Poisson process $(N_t)_{t\geq 0}$ of claim arrivals, and i.i.d. strictly positive claim sizes $\{S_i, i\in\NN\}$ independent of $(N_t)_{t\geq 0}$. The time $\tau(y)$ is the ruin time. In this case, it is well-known that $(R_t^y)_{t\geq 0}$ is dual in the \AL original sense of Siegmund~\cite{Siegmund1976} (i.e.~\eqref{dual} holds for all $x,y,t\geq 0$, not $x,y\in \bR$) \normal to the M/G/1-queue workload process $(X_t^x)_{t\geq 0} = (V_t^x)_{t\geq 0}$ defined by
		$$V_t^x= x+ \sum_{i=1}^{N_t} S_i - \int_{\HeHe (0,t]} c \mathds{1}_{\{V_s^x>0\}} \diff s, \quad x,t\geq 0,$$
		see \cite{Prabhu} for the earliest reference on this relation, or \cite[Chapter III.2]{AsmussenAlbrecher} for a textbook treatment.  (Actually, apart from~\cite[Example~3.5]{Asmussen95}, where no proof is given, in most references equation~\eqref{dual} is only verified for the case $y\geq 0$ and $x=0$. However, the proof given in \cite[Chapter III.2]{AsmussenAlbrecher} can be easily adapted for general $x,y\geq 0$).
		Thus, in the present situation, the ruin probability can be written as
		\begin{equation*}
			\PP(\tau(y)<\infty) = \lim_{t\to \infty} \PP(R_t^y \leq 0) =  \lim_{t\to \infty}\PP(V_t^0\geq y ) = \PP(V\geq y),
		\end{equation*}
		where $V$ is a generic random variable whose law is given by the limiting distribution of the dual process $(V_t^0)_{t\geq 0}$ which is assumed to exist 
		and to have no atom at $y$, cf. \cite[Chapter III.2]{AsmussenAlbrecher}.  \\
		Hence, in this situation, duality allows to transform the hitting (or ruin) probability of $(R_t^y)_{t\geq 0}$ into the cdf of the stationary distribution of $(V_t^x)_{t\geq 0}$, and vice versa. This fact has been used to derive various results in both risk theory and queueing theory; see \cite[Chapter VII.7]{AsmussenAlbrecher} and references therein.		
		Embedding the Cramér-Lundberg process $(R_t^y)_{t\geq 0}$ into an investment market whose dynamics are driven by a Lévy process, say $(W_t)_{t\geq 0}$, yields a capital process that is often referred to as \emph{Paulsen's risk model} \cite{Paulsen}. This is a risk process $(R_t^y)_{t\geq 0}$ whose dynamics are described by the SDE
		\begin{equation} \label{eq-SDE-Paulsen}
			\diff R_t^y = R_{t-}^y \diff W_t + \diff K_t, \quad t\geq 0, \quad R_0^y = y,
		\end{equation}
		which turns out to be of the form \eqref{eq:GOUSDE}.
	\end{example}

	Assuming that  the driving processes have finite jump activity, in \cite{Paulsen} the solution of \eqref{eq-SDE-Paulsen} is studied and also an expression for its ruin probability is derived. This result of Paulsen has been generalized to the setting of this paper by Bankovsky and Sly in \cite{BankSly} and then reads as follows.
	
	\begin{proposition}[{\cite[Thm. 4]{BankSly}}]\label{p-GOU-ruin2}
		Suppose that \eqref{eq:assU(B)} is satisfied and consider the GOU process $(V_t^x)_{t\geq 0}$ solving \eqref{eq:GOUSDE}.
		Assume that $\lim_{t\to\infty} \cE(U)_t^{-1}=0$ almost surely and that the integral $\int_{(0,t]} \cE(U)_{s-}^{-1} \, \di \eta_s$ converges almost surely as $t\to\infty$  to a non-degenerate finite random variable $\int_{\HeHe (0,\infty)} \cE(U)_{s-}^{-1} \, \di \eta_s$ with
		distribution function $H$. Let $\tau(x)= \inf\{t\geq 0, V_t^x \leq 0\}$. Then
		$$\PP(\tau(x)<\infty) \, \EE[H(-V_{\tau(x)}^x)|\tau(x)<\infty]= H(-x), \quad x\geq 0, $$
		where $\EE[\cdot| \tau(x)<\infty]$ is interpreted as $0$ if $\PP(\tau(x)<\infty)=0$.
	\end{proposition}
	
	\section{The dual of the generalized Ornstein-Uhlenbeck process} \label{S-dual}
	\setcounter{equation}{0}
	
	\HeHe Having established the necessary background we now turn to our main results, which will be presented in this and in the forthcoming Section \ref{S-timereverse}. \normal
	\AL	In particular, we will characterize when GOU processes have a dual process, and show that GOU processes are self-dual in the sense that the dual process (when existent) is again a GOU process. Similar phenomena have already been observed in special settings for some $[0,\infty]$-valued Markov processes, see e.g.~\cite[Example 3.5]{Asmussen95}. \normal

%
%
%

	\subsection{\texorpdfstring{Siegmund duality on $\RR$}{Siegmund duality on R}}

%

\AL For the ease of exposition, let us recall the definition of Siegmund-duality on the whole real line, as it was given at the beginning of the introduction. \normal

	\begin{definition}\label{def-dualonR}
		Let $(X_t^x)_{t\geq 0}$ and $(Y_t^y)_{t\geq 0}$ be (universal) time-homogeneous Markov processes on $\RR$, where $x,y\in \bR$. Then $(Y_t^y)_{t\geq 0}$ is said to be \emph{dual (in the sense of Siegmund) to $(X_t^x)_{t\geq 0}$} if 
		for all $x,y\in \RR$
		\begin{equation}\label{dualonR}
			\PP(X_t^x \geq y) = \PP(Y_t^y \leq x), \quad 0\leq t<\infty.
		\end{equation}
	\end{definition}
	
\AL Before we treat duality for GOU processes, let us study some general consequences of the definition. \normal
Note that it follows immediately from \eqref{dualonR} that if $(Y_t^y)_{t\geq 0}$ is dual to $(X_t^x)_{t\geq 0}$, then $(-X_t^x)_{t\geq 0}$ is dual to $(-Y_t^y)_{t\geq 0}$.
	However, duality as defined is not necessarily a symmetric relation as illustrated by the following example.
	
	 \begin{example}
		Let $(N_t)_{t\geq 0}$ be a homogeneous Poisson process (or any other non-trivial subordinator) and define the time-homogeneous Markov processes $(X_t^x)_{t\geq 0}$ and $(Y_t^y)_{t\geq 0}$ by
		\begin{align*}
			X_t^x=\begin{cases}
				x,& X_0^x=x\geq 0,\\
				x-N_t, & X_0^x=x<0,
			\end{cases}\qquad Y_t^y = \begin{cases}
				y, & Y_0^y=y\geq 0,\\
				\min\{y+N_t, 0\}, & Y_0^y=y<0.
			\end{cases}
		\end{align*}
		Then it follows immediately that
		\begin{align*}
			\PP(X_t^x\geq y) = \left.\begin{cases}
				\mathds{1}_{\{x\geq y\}}, & x\geq 0,\\
				\PP(N_t\leq x-y), & x<0
			\end{cases}\right\} = \PP(Y_t^y\leq x),
		\end{align*}
		and hence $(Y_t^y)_{t\geq 0}$ is dual to $(X_t^x)_{t\geq 0}$ in the sense of Definition~\ref{def-dualonR}. However we note that for $y<0$
		\begin{align*}
			\PP(Y_t^y \geq 0) &= \PP(y+N_t\geq 0) \not= 0 = \PP(X_t^0\leq y),
		\end{align*}
		and hence $(X_t^x)_{t\geq 0}$ is not dual to $(Y_t^y)_{t\geq 0}$.
	\end{example}
	
	 \begin{remark}
	 	From the definition, one can easily derive conditions for symmetry of the duality relation as follows.
	 	\begin{enumerate}
	 		\item Suppose that $(Y_t^y)_{t\geq 0}$ is dual to $(X_t^x)_{t\geq 0}$. Then a necessary and sufficient condition for $(X_t^x)_{t\geq 0}$ to be dual to~$(Y_t^y)_{t\geq 0}$ is that
	 		\begin{equation} \label{eq-dual-nec}
	 			\PP(X_t^x = y) \, = \, \PP( Y_t^y = x ) \quad \forall\; x,y\in \RR, \; t \geq 0,\end{equation}
	 		as is easily seen from \eqref{dualonR}.
	 		\item 	Provided that $(Y_t^y)_{t\geq 0}$ is dual to $(X_t^x)_{t\geq 0}$, a sufficient (but not necessary) condition for \eqref{eq-dual-nec} to hold is that
	 		$\PP(X_t^x\geq y)$ is continuous as a function in $x$ for each fixed $y\in \bR$ and $t>0$, and continuous as a function in $y$ for each fixed $x\in \RR$ and $t>0$.
	 	\end{enumerate}
	\end{remark}
	
    \begin{remark}\label{rem-monotonicity}
	   \AL As is the case in the original setting of Siegmund duality~\cite{Siegmund1976}, \normal where $[0,\infty]$-valued Markov processes were considered, it is trivial from \eqref{dualonR} that a \emph{necessary} condition for an $\RR$-valued Markov process $(X_t^x)_{t\geq 0}$ to have a dual process as described in Definition~\ref{def-dualonR} is that the function $x\mapsto \PP (X_t^x \geq y)$ is nondecreasing and right-continuous for all $y\in \RR$ and $t\geq 0$.
    \end{remark}

	The next theorem characterises when a dual of a GOU process $(V_t^x)_{t\geq 0}$ exists and shows that then the dual is again a GOU process $(R_t^y)_{t\geq 0}$, and that then $(V_t^x)_{t\geq 0}$ is also dual to $(R_t^y)_{t\geq 0}$.
	
	\begin{theorem}\label{thm:dualGOU}
		Assume \eqref{eq:assU(A)} and let $(V_t^x)_{t\geq 0}$ be the GOU process given by \eqref{eq:GOU-explicit}.  Then a dual process to $(V_t^x)_{t\geq 0}$ in the sense of Definition~\ref{def-dualonR} exists if and only if \eqref{eq:assU(B)} is satisfied. In that case, the GOU process $(R_t^y)_{t\geq 0}$ solving the SDE
		\begin{equation}
			\diff R_t^y = R_{t-}^y \diff W_t + \diff K_t, \quad R_0^y=y, \label{eq:GOUSDE2}
		\end{equation}
		where $(W_t,K_t)_{t\geq 0}$ is another bivariate L\'evy process defined by
		\begin{eqnarray*}
			\begin{pmatrix}
				W_t \\ K_t
			\end{pmatrix}
			= \begin{pmatrix}
				- U_t  + \sum_{0<s\leq t} \frac{(\Delta U_s)^2}{1+\Delta U_s}+ t \sigma_U^2  \\ -\eta_t
			\end{pmatrix}, \quad t\geq 0,
		\end{eqnarray*}
		is dual to $(V_t^x)_{t\geq 0}$.
        Moreover, $(V_t^x)_{t\geq 0}$ is dual to $(R_t^y)_{t\geq 0}$.
	\end{theorem}
	
	\begin{proof}
		If \eqref{eq:assU(B)} is violated, then $(V_t^x)_{t\geq 0}$ is not stochastically monotone by Lemma~\ref{lem-GOUstochmon}, hence no dual process can exist by Remark~\ref{rem-monotonicity}. So suppose from now on that \eqref{eq:assU(B)} is satisfied. Then $\cE(U)_t>0$ for all $t\geq 0$ by \eqref{Dolean}. Thus, using \eqref{eq:GOU-explicit},
		\begin{align}
			V_t^x \geq y \quad &\Leftrightarrow \quad \cE(U)_t \bigg( x+ \int_{(0,t]} \cE(U)_{s-}^{-1} \diff \eta _s\bigg) \geq y \nonumber \\
			&\Leftrightarrow \quad   \cE(U)_t^{-1} y -  \int_{(0,t]} \cE(U)_{s-}^{-1} \diff \eta _s \leq x. \label{eq:cdfGOU}
		\end{align}
		Further, by Lemma \ref{lemma:BLMLemma3.4},
		\begin{equation} \label{eq:DDEinverse}
			\cE(U)_t^{-1}= \cE(W)_t, \quad t\geq 0,
		\end{equation}
		for the L\'evy process $(W_t)_{t\geq 0}$ given by
			\begin{equation} \label{eq:WviaU}
			W_t= - U_t  + \sum_{0<s\leq t} \frac{(\Delta U_s)^2}{1+\Delta U_s} + t \sigma_U^2 ,
				\end{equation}
		such that in particular $\Delta W_t = \frac{-\Delta U_t}{1+\Delta U_t}>-1$. Moreover, for every $t\geq 0$, by Lemma \ref{lemma:BLMLemma3.1}, 
		\begin{align} \label{eq:BLMLemma3.1}
			\begin{pmatrix}
				\cE(W)_t \\ \int_{(0,t]} \cE(W)_{s-} \diff  \eta_s
			\end{pmatrix} \overset{d}= \begin{pmatrix}
				\cE(W)_t \\ \cE(W)_t \int_{(0,t]} \cE(W)_{s-}^{-1} \diff Z_s
			\end{pmatrix} ,
		\end{align}
		for some L\'evy process $(Z_t)_{t\geq 0}$ that can be computed via \eqref{eq:etaviaUL} as 
		\begin{align*}
			Z_t &= \eta_t - \sum_{0<s\leq t} \frac{\Delta W_s \Delta \eta_s}{1+\Delta W_s} - t\sigma_{W,\eta}\\
			&=  L_t - \sum_{0<s\leq t} \frac{\Delta U_s \Delta L_s}{1+ \Delta U_s} - t \sigma_{U,L} - \sum_{0<s\leq t} \frac{\frac{-\Delta U_s}{1+\Delta U_s} \frac{\Delta L_s}{1+\Delta U_s}}{1+\frac{-\Delta U_s}{1+\Delta U_s}} + 		t\sigma_{U,L} = L_t,
		\end{align*}
		via \eqref{eq:WviaU} and \eqref{eq:etaviaUL}, from which we also derived that $\sigma_{W,\eta}=-\sigma_{U,\eta}=-\sigma_{U,L}$. Thus, inserting \eqref{eq:DDEinverse} in \eqref{eq:cdfGOU}, and applying \eqref{eq:BLMLemma3.1} we observe
		\begin{align*}
			\PP(V_t^x \geq y)& = \PP\bigg( \cE(W)_t y -  \int_{(0,t]} \cE(W)_{s-} \diff \eta _s \leq x\bigg) \nonumber \\
			&= \PP \bigg( \cE(W)_t y - \cE(W)_t \int_{(0,t]} \cE(W)_{s-}^{-1} \diff L_s \leq x \bigg) \nonumber \\
			&= \PP (R_t^y \leq x)
		\end{align*}
		for the GOU process
		\begin{equation} \label{eq:dualGOUexplicit}
			R_t^y =  \cE(W)_t \bigg(y +  \int_{(0,t]} \cE(W)_{s-}^{-1} \diff (-L_s) \bigg), \quad t\geq 0.
		\end{equation}
		Defining the auxiliary Lévy process
		$$-\xi_t = \log \cE(W)_t = W_t - \frac12 \sigma_W^2 t + \sum_{0<s\leq t} \bigg( -\Delta W_s + \log(1+\Delta W_s)\bigg),\quad t \geq 0,$$
		using \eqref{Dolean}, we obtain that the SDE solved by $(R_t^y)_{t\geq 0}$ is hence of the form \eqref{eq:GOUSDE2} with $(W_t)_{t\geq 0}$ as given in \eqref{eq:WviaU} and $(K_t)_{t\geq 0}$ following via \eqref{eq:BLM(1.3)} and \eqref{eq:WviaU} as
		\begin{align*}
			K_t&= -L_t + \sum_{0<s\leq t}  (\re^{-\Delta \xi_s} -1) (-\Delta L_s) - t \sigma_{\xi,-L}\\
			&= -L_t + \sum_{0<s\leq t} \frac{\Delta U_s\Delta L_s}{1+\Delta U_s} + t \sigma_{U,L} = - \eta_t,
		\end{align*}
		since $\sigma_{\xi,-L} = \sigma_{-W,-L} = \sigma_{U,-L} = - \sigma_{U,L}$. As we have shown that $\PP(V_t^x\geq y) = \PP(R_t^y\leq x)$ for all $x,y\in \RR$ and $t\geq 0$, the process $(R_t^y)_{t\geq 0}$ is dual to $(V_t^x)_{t\geq 0}$. Similar calculations, or an application of the obtained result on $(R_t^y)_{t\geq 0}$, show that $\PP(V_t^x= y) = \PP(R_t^y= x)$ so that also $(V_t^x)_{t\geq 0}$ is dual to $(R_t^y)_{t\geq 0}$ by \eqref{eq-dual-nec}.
	\end{proof}

	We can relate the stationary distribution of the GOU $(V_t^x)_{t\geq 0}$ to that of its dual process $(R_t^y)_{t\geq 0}$, provided they exist.

	\begin{corollary} \label{c-dual-stationary}
		Suppose that \eqref{eq:assU(B)} is satisfied, let $(V_t^x)_{t\geq 0}$ be the GOU process given by \eqref{eq:GOU-explicit}, and $(R_t^y)_{t\geq 0}$ its dual process as specified in \eqref{eq:GOUSDE2}.
		\begin{enumerate}[label=(\alph*)]
		\item Suppose $\lim_{t\to\infty} \cE(U)_t=0$ almost surely and that $V_\infty := \int_{\HeHe (0,\infty)} \cE(U)_{s-} \, \diff L_s$ exists almost surely and is finite. Then the distribution of $V_\infty$ is both the stationary distribution of the causal GOU process $(V_t^x)_{t\geq 0}$ as well as the stationary distribution of the non-causal GOU process $(R_t^y)_{t\geq 0}$.
		\item Suppose $\lim_{t\to\infty} \cE(U)_t^{-1}=0$ almost surely and that $R_\infty := \int_{\HeHe (0,\infty)} \cE(U)_{s-}^{-1} \, \diff K_s$ exists almost surely and is finite. Then the distribution of $R_\infty$ is both the stationary distribution of the non-causal GOU process $(V_t^x)_{t\geq 0}$ as well as the stationary distribution of the causal GOU process $(R_t^y)_{t\geq 0}$.
		\end{enumerate}
	\end{corollary}

	\begin{proof}
		Recall \eqref{eq:DDEinverse} from which $\lim_{t\to\infty}\cE(W)_t^{-1} = \lim_{t\to\infty}\cE(U)_{t}$. Further, the right-hand side of \eqref{eq:etaviaUL} applied to $(W_t,K_t)_{t\geq 0}$ instead of $(U_t,L_t)_{t\geq 0}$ gives
		$$K_t - \sum_{0<s\leq t} \frac{\Delta W_s \Delta K_s}{1+\Delta W_s} - t \sigma_{W,K} = -L_t ,$$
		where we used the definition of $K_t$, and Equations \eqref{eq:WviaU} and \eqref{eq:etaviaUL}, from which $\Delta W_s = \frac{-\Delta U_s}{1+\Delta U_s}$,  $\Delta K_s = -\Delta \eta_s = \frac{-\Delta L_s}{1+\Delta U_s}$, and $\sigma_{W,K} = \sigma_{U,L}$. Assertion (a) is then an immediate consequence of Lemma~\ref{lemma:BLMThm2.1}. The proof of (b) is similar.
	\end{proof}

	As motivated in Example \ref{ex:risk} we aim to investige the connection between the limiting distribution of $(R_t^y)_{t \ge 0}$ and the ruin probability of $(V_t^x)_{t \ge 0}$. This can now be done by combining Proposition~\ref{p-GOU-ruin2} with our results.
	
	\begin{corollary}\label{cor-GOUhitting}
		Under the conditions and with the same notations as in Proposition \ref{p-GOU-ruin2},
		the limit of $\PP(V_t^x\leq 0)$ as $t\to\infty$ exists and is equal to $H(-x) = \PP(R_\infty \geq x)$, so that
		\begin{equation} \label{eq-BaSly1}
			\lim_{t\to\infty} \PP(V_t^x \leq 0) = \PP(R_\infty \geq x) =
			\PP(\tau(x)<\infty) \, \EE[H(-V_{\tau(x)}^x)|\tau(x)<\infty], \quad  x\geq 0,
		\end{equation}
		where $R_\infty = \int_{\HeHe (0,\infty)} \cE(U)_{s-}^{-1} \, \di K_s$ as in Corollary~\ref{c-dual-stationary}~(b).
		In particular, if $-L$ is a subordinator, then
		$$ \PP(\tau(x)<\infty) = \PP(R_\infty\geq x) = H(-x), \quad x \geq 0. $$
	\end{corollary}
	
	\begin{proof}
		 By Theorem~\ref{thm:dualGOU}, not only is $(R_t^y)_{t\geq 0}$ dual to $(V_t^x)_{t\geq 0}$, but also  $(V_t^x)_{t\geq 0}$ is dual to $(R_t^y)_{t\geq 0}$. Hence $\PP(V_t^x\leq 0) = \PP(R_t^0 \geq x)$.
		Since the distribution of $R_\infty = \int_{\HeHe (0,\infty)} \cE(U)_{s-}^{-1} \, \di K_s = - \int_{\HeHe (0,\infty)} \cE(U)_{s-}^{-1} \, \di \eta_s$ is the stationary distribution of a causal GOU process which  by assumption is non-degenerate, it cannot have atoms by \cite[Thm.~2.2]{BertoinAtoms}. Since $R_t^0$ converges in distribution to $R_\infty$ as $t\to\infty$ by Lemma~\ref{lemma:BLMThm2.1}~(a), we conclude that $\PP(V_t^x\leq 0)$ converges to $\PP(R_\infty \geq x) = \PP(-R_\infty \leq -x) = H(-x)$.
		The formula \eqref{eq-BaSly1} for the limit is then immediate from Proposition \ref{p-GOU-ruin2}.
		
		For the second statement it suffices to assume that $\PP(\tau(x) < \infty) > 0$,  for if $\PP(\tau(x) < \infty) = 0$ then also $\PP(R_\infty \geq x) = 0$ by \eqref{eq-BaSly1} and the claim follows.
		Note that if $-L$ is a subordinator, also $-\eta$ is a subordinator  by the observation just below~\eqref{eq:GOU-explicitxi}, which implies that $\int_{\HeHe (0,\infty)} \cE(U)_{s-}^{-1} \, \di \eta_s\leq 0$ a.s. and hence $H(0)=1$.
		Since $V_{\tau(x)}^x \leq 0$ on $\{\tau(x) < \infty\}$ by the c\`adl\`ag paths of $(V^x_t)_{t\geq 0}$ this implies $H(-V_{\tau(x)}^x) =1$ and hence the statement when $\PP(\tau(x) < \infty) > 0$ by \eqref{eq-BaSly1}.
	\end{proof}
	
	\subsection{Non-negative Siegmund dual processes} \label{S-dualhalfline}

\AL As mentioned in Section \ref{SPrelimGOU}, if $\Delta U_t>-1$ for all $t$ and $L$ is a subordinator, then $(V_t^x)_{t\geq 0}$ (with $x\geq 0$) has state-space $[0,\infty)$, and hence this GOU process not only fits into the setting of Definition~\ref{def-dualonR}, but also
into the setting originally considered in Siegmund~\cite{Siegmund1976}, where $[0,\infty]$-valued Markov processes were studied.
Thus, according to Siegmund's original work, this GOU process has a dual process on $[0,\infty]$ in the sense that \eqref{dualonR} holds for all starting values $x,y\geq 0$. The precise form of this process will be derived in the next proposition. \normal
	
	\begin{proposition}\label{prop-dualhalfline}
		Assume \eqref{eq:assU(B)} and let $(V_t^x)_{t\geq 0}$ with $x\geq 0$ be the GOU given by \eqref{eq:GOU-explicit}. Further assume that  $L$ is a subordinator. Let $(R_t^y)_{t\geq 0}$ with $y\geq 0$ be the solution of \eqref{eq:GOUSDE2} and set
		$$\tau_{R}(y):= \inf\{t\geq 0: R_t^y\leq 0\}. $$
		Then the process $(\widehat{R}_t^y)_{t\geq 0}$, $y\geq 0$, given by
		$$\widehat{R}_t^y:= R_t^y \mathds{1}_{\{t<\tau_{R}(y)\}},$$
		is a time-homogeneous Markov process on $[0,\infty)$, and it is dual to $(V_t^x)_{t\geq 0}$, $x\geq 0$, in the sense that \AL \eqref{dualonR} holds for all $x,y\geq 0$ and $t\geq 0$. \normal Moreover, we have $\widehat{R}_t^y = \max\{ R_t^y, 0\}$ for $t,y\geq 0$.
	\end{proposition}
	
	\begin{proof}
		Recall from Theorem \ref{thm:dualGOU} and its proof that $(R_t^y)_{t\geq 0}$ is dual to $(V_t^x)_{t\geq 0}$ in the sense of \eqref{dualonR}, i.e.
		$$\PP(V_t^x\geq y) = \PP(R_t^y \leq x), \quad \forall t\geq 0, \; \forall x,y\in\RR,$$
		and that $R_t^y$ is explicitly given by \eqref{eq:dualGOUexplicit}. As $(R_t^y)_{t\geq 0}$ -- being a GOU process -- is a time-homogeneous Markov process, its killed version $(\widehat{R}_t^y)_{t\geq 0}$ is a time-homogeneous Markov process as well. Moreover, under the given conditions, $\cE(W)_t=\cE(U)_t^{-1}>0$ for all $t\geq 0$, and as $L$ is a subordinator it follows from \eqref{eq:dualGOUexplicit} that
		$$y\leq 0\quad \Rightarrow \quad R_t^y\leq 0 ,\; \forall t\geq 0.$$
		Due to time-homogeneity of $(R_t^y)_{t\geq 0}$ this implies $R_t^y\leq 0$ for any $y\in \RR$, and for all $t\geq \tau_R(y)$. In particular
		$$ \widehat{R}_t^y= R_t^y \mathds{1}_{\{t<\tau_{R}(y)\}} = \max\{R_t^y, 0\}, \quad t\geq 0.$$
		Thus, for all $t\geq 0$ and all $x,y\geq 0$
		\begin{align*}
			\PP(\widehat{R}_t^y\leq x) &= \PP(\max\{R_t^y, 0\}\leq x)
			= \PP(R_t^y\leq x)  = \PP(V_t^x\geq y),
		\end{align*}
		which implies the statement.
	\end{proof}

	\begin{remark}	
		Observe that although $(V_t^x)_{t\geq 0}$, $x\in \bR$, is dual to $(R_t^y)_{t\geq 0}$, $y\in \bR$, in the sense of \eqref{dualonR} by Theorem~\ref{thm:dualGOU}, in the situation of Proposition~\ref{prop-dualhalfline} the process $(V_t^x)_{t\geq 0}$, $x\geq 0$, is in general \emph{not dual} to $(\widehat{R}_t^y)_{t\geq 0}$, $y\geq 0$, \AL in the sense that \eqref{dualonR} holds for all $x,y\geq 0$. \normal This can be seen from
		$$\PP(\widehat{R}_t^0 \geq 0)=1 \neq \PP\left( \cE(U)_t \int_{(0,t]} \cE(U)_{s-}^{-1} \diff \eta_s\leq 0 \right) = \PP(V_t^0\leq 0) \quad\forall\; t>0,$$
		provided $\eta$ is not the zero-subordinator. Since the function $[0,\infty) \to [0,1]$, $y\mapsto \PP(\widehat{R}_t^y\geq x)$, is nondecreasing and right-continuous for any $t,x\geq 0$, by Siegmund's result \AL (\cite[Thm. 1]{Siegmund1976}) \normal the process $(\widehat{R}_t^y)_{t\geq 0}$, $y\geq 0$, has a dual on $[0,\infty]$, but it is not $(V_t^x)_{t\geq 0}$, but instead the process $(\widehat{V}_t^x)_{t\geq 0}$ defined as
		$$ \widehat{V}_t^x= V_t^x \mathds{1}_{\{t<\tau_{V}(x)\}}, \quad\text{with } \tau_{V}(x):= \inf\{t\geq 0: V_t^x\leq 0\}.$$
		To see this, observe that $\widehat{V}_t^0\equiv 0$ which implies
		$$\PP(\widehat{R}^y_t\geq 0)=1=\PP(\widehat{V}_t^0\leq y)\quad \text{for all } y\geq 0, t\geq 0.$$
		Further, for $x>0$, as $L$ and hence $\eta$ is a subordinator, $V_t^x>0$ for all $t\geq 0$, implying $\tau_V(x)=\infty$ almost surely. With this it follows from Theorem \ref{thm:dualGOU} and Proposition \ref{prop-dualhalfline} for all $y\geq 0, x>0$ and $t\geq 0$
		\begin{align*}
			\PP(\widehat{R}^y_t\geq x) &= \PP(\max\{R_t^y,0\}\geq x) = \PP(R_t^y\geq \HeHe x\normal)\\
			&= \PP(V_t^x\leq y) = \PP(V_t^x \mathds{1}_{\{t<\tau_{V}(x)\}} \leq y)  = \PP(\widehat{V}_t^x \leq y).
		\end{align*}
		Hence $(\widehat{V}_t^x)_{t\geq 0}$ is dual to $(\widehat{R}_t^y)_{t\geq 0}$ in the sense \AL that \eqref{dualonR} holds for all $x,y\geq 0$, \normal but again the converse statement fails as can easily be checked considering the duality relation for $x=0$.
	\end{remark}

	As a consequence of Proposition~\ref{prop-dualhalfline} we can relate the ruin probability of $(R_t^y)_{t\geq 0}$ to the limit distribution of $(V_t^0)_{t\geq 0}$ in the subordinator case.

	\begin{corollary} \label{c-subordinator-ruin}
		Assume \eqref{eq:assU(B)} and let $(V_t^x)_{t\geq 0}$, $(R_t^y)_{t\geq 0}$ and $\tau_R(y)$  be as in Proposition~\ref{prop-dualhalfline} with $L$ being a subordinator. Then
		$$\PP(\tau_R(y) < \infty) = \lim_{t\to\infty} \PP(V_t^0 \geq y)\quad \text{for every }y\geq 0.$$
		Denote $V_\infty := \int_{\HeHe (0,\infty)} \cE(U)_s \, \diff L_s$, which is a $[0,\infty]$-valued numerical random variable (in case of almost sure convergence, it is finite). Then, if neither $U$ nor $L$ is the zero-process and if there exists no $k\in\RR\setminus \{0\}$ such that \eqref{eq-degenerate} holds, then
		$$\PP(\tau_R(y) < \infty) = \PP(V_\infty \geq y)$$
		for every $y\geq 0$.
		Finally, if additionally $\lim_{t\to\infty} \cE(U)_t =0$ almost surely, then
		$$\PP(\tau_R(y) < \infty) = \lim_{t\to\infty} \PP(V_t^x \geq y)$$
		for all $x,y \geq 0$.
	\end{corollary}

	\begin{proof}
		Observe that $(\widehat{R}_t^y)$ as defined in Proposition~\ref{prop-dualhalfline} stays at 0 once it has reached 0. Since $(\widehat{R}_t^y)_{t\geq 0}$ is dual to $(V_t^x)_{t\geq 0}$ in the sense of \eqref{dual}, for fixed $y\geq 0$ we have
		\begin{align*}
		\PP(\tau_R(y) < \infty) & = \PP (\exists\, t \geq 0 : \widehat{R}_t^y = 0)  = \lim_{t\to\infty} \PP(\widehat{R}_t^y \leq 0) = \lim_{t\to\infty} \PP(V_t^0 \geq y).
		\end{align*}
		Let $\xi_t = - \log \cE(U)_t$ as in \eqref{eq:GOU-explicitxi}, then according to \cite[Thm. 2.1]{EricksonMaller2005}, $\int_{(0,t]} \re^{-\xi_{s-}} \diff L_s = \int_{(0,t]} \cE(U)_{s-}^{-1} \, \di L_s$ either converges a.s. to a finite random variable as $t\to\infty$, or diverges to $\infty$ in probability (and hence almost surely since $L$ is a subordinator), or \eqref{eq-degenerate2} holds for some $k\neq 0$. But \eqref{eq-degenerate2} is equivalent to \eqref{eq-degenerate} which has
		been ruled out by assumption. Hence $V_\infty$ is either finite almost surely or almost surely equal to $+\infty$.  When $V_\infty$ is finite almost surely, its distribution cannot have atoms by \cite[Thm.~2.2]{BertoinAtoms} since \eqref{eq-degenerate} has been ruled out. Hence, in all cases the distribution of $V_\infty$ has no atoms on the real line.\\
		 From \eqref{eq:GOU-explicit} and Lemma~\ref{lemma:BLMLemma3.1} we see that $V_t^x$ has the same distribution as $\cE(U)_t x + \int_{(0,t]} \cE(U)_{s-} \, \diff L_s$, which converges almost surely to $V_\infty$ as $t\to\infty$ if $x=0$, or if $x>0$ and additionally  $\cE(U)_t$ converges to $0$ as $t\to\infty$. Since $V_\infty$ has no atom at $y$, we then have $\lim_{t\to\infty} \PP(V_t^x \geq y) = \PP(V_\infty \geq y)$ in both cases, giving the claim.
	\end{proof}

	\section{The inverse flow of the generalized Ornstein-Uhlenbeck process} \label{S-timereverse}
	\setcounter{equation}{0}
	
	In contrast to duality, which can be considered a distributional concept, time-reversals of stochastic processes are typically defined pathwise. Therefore, we use a different framework in this section than before and study the GOU process in terms of its stochastic flow and the inverse stochastic flow. This is possible because, in the GOU setting, the structure of the SDE admits an explicit inverse mapping for the jumps, rather than applying the generalized inverse used in \cite{AsmussenSigman1996}. In comparison with the time-reversed SDE that can be derived via \cite[Thm. VI.4.22]{Protter2003}, in Theorem \ref{thm:time_reversed_GOU}  we obtain the same SDE for the inverse stochastic flow. Similar considerations appear at the end of Chapter VI.4 in \cite[Proof of Thm. VI.4.23]{Protter2003}, where time-reversals of SDEs of Itô diffusions are considered. \\
	Recall the definition of the \emph{time-reversal} $\widetilde{X} \HeHe =(\widetilde{X}_s)_{s\in [0,t]}\normal$ at $t>0$ of some càdlàg process $X=(X_s)_{s\in [0,t]}$ as
	\begin{equation}\label{eq:deftimereverse}
		\widetilde{X}_s = \begin{cases}
			0, & \text{if }s=0, \\
			X_{(t-s)-} - X_{t-}, & \text{if }0<s<t, \\
			X_0 - X_{t-}, & \text{if }s=t.
		\end{cases}
	\end{equation}
	For $t=1$ this coincides with the definition given in \cite[Chapter VI.4]{Protter2003}. It is well known that when $X$ is a L\'evy process, then $(\widetilde{X}_s)_{s\in [0,t]}$ is also a L\'evy process, with the same distribution as $(-X_s)_{s\in [0,t]}$, see e.g.~\cite[Lem.~II.2]{Bertoin1996}, \cite[Thm.~11.4]{BrockwellLindner2024} or \cite[Prop.~41.8]{Sato1999}. Stochastic integrals with respect to time-reversed L\'evy processes are then to be seen with respect to their augmented natural filtration, for which the time-reversed L\'evy processes are semimartingales.	

	As before, let $(V_t^x)_{t\geq 0}$ be a solution of the SDE \eqref{eq:GOUSDE} for the bivariate L\'evy process $(U_t,L_t)_{t\geq 0}$. As no stochastic monotonicity is needed in the context of time-reversals, we shall only assume \eqref{eq:assU(A)}, i.e. $\Delta U_t\neq -1$ for all $t$.
		
	\begin{theorem}\label{thm:time_reversed_GOU}
		Let $(V_s^x)_{s\geq 0}$, $x\in\RR$, be a solution of the SDE~\eqref{eq:GOUSDE} for the bivariate L\'evy process $(U_s,L_s)_{s\geq 0}$ fulfilling \eqref{eq:assU(A)}, as given in \eqref{eq:GOU-explicit} with $\eta$ defined in \eqref{eq:etaviaUL}. Let $t>0$  be fixed.
		Denote the time-reversal of $(U,L,\eta)$ on $[0,t]$ by $(\widetilde{U}, \widetilde{L}, \widetilde{\eta})$, and define
		$$T_s := \widetilde{U}_s + \sigma_U^2 s + \sum_{0<r\le s} \frac{(\Delta \widetilde{U}_r)^2}{1-\Delta \widetilde{U}_r}, \quad 0\leq s\leq t.$$
		Then $(T_s,\widetilde{L}_s, \widetilde{\eta}_s)_{s\in [0,t]}$ is a L\'evy process on $[0,t]$ and, almost surely, for all $s\in [0,t]$
		\begin{equation}\label{eq-flow-0}
			V_{(t-s)-}^x = \cE(T)_s \left(V_t^x + \int_{(0,s]} \cE(T)_{u-}^{-1} \, \di \widetilde{L}_u\right).
		\end{equation}
		Further, the stochastic process corresponding to the inverse stochastic flow induced by \eqref{eq-flow-0}, i.e. the process $(R_s^y)_{s\in [0,t]}$, $y\in \bR$, defined by
		\begin{equation} \label{eq-flow-1}
			R_s^y := \cE(T)_s \left(y + \int_{(0,s]} \cE(T)_{u-}^{-1} \diff \widetilde{L}_u\right),
		\end{equation}
		is the GOU process which is the unique solution of the SDE
		\begin{equation} \label{eq:time reversed GOU}
			\diff R_s^y = R_{s-}^y \diff T_s + \diff \widetilde{\eta}_s, \quad s\in [0,t], \quad R_0^y = y.
		\end{equation}
	\end{theorem}
			
	\begin{proof}
		Since $(U,L,\eta)$ is a L\'evy process, so is its time-reversal $(\widetilde{U}_s,\widetilde{L}_s,\widetilde{\eta}_s)_{s\in [0,t]}$ on $[0,t]$, hence also $(T_s,\widetilde{L}_s,\widetilde{\eta}_s)_{s\in [0,t]}$ is a L\'evy process.
		Observe that $\Delta \widetilde{U}_s = - \Delta U_{t-s}$ for $0<s<t$, and similarly for $L$ and $\eta$.
		Since $(U,L,\eta)$ almost surely does not jump at the fixed time $t$, we assume that $(\Delta U_t,\Delta L_t,\Delta \eta_t) = 0$ and hence $V_{t}^x = V_{t-}^x$ everywhere. From \eqref{eq:GOU-explicit} we obtain
		\begin{equation} \label{eq-flow-expl}
			V_t^x = \frac{\cE(U)_t}{\cE(U)_{t-s}} \left( V_{t-s}^x + \int_{(t-s,t]}
			\left(\frac{\cE(U)_{u-}}{\cE(U)_{t-s}}\right)^{-1} \diff \eta_u\right).
		\end{equation}
		Solving this for $V_{t-s}^x$ results in
		\begin{equation*}
			V_{t-s}^x = \frac{\cE(U)_{t-s}}{\cE(U)_t} \left( V_t^x - \frac{\cE(U)_t}{\cE(U)_{t-s}}
			\int_{(t-s,t]}
			\left( \frac{\cE(U)_{u-}}{\cE(U)_{t-s}} \right)^{-1} \, \di \eta_u \right)
		\end{equation*}
		and taking the c\`adl\`ag version $V_{(t-s)-}^x$ for $s\mapsto V_{t-s}^x$ yields
		\begin{equation} \label{eq-flow-rev}
			V_{(t-s)-}^x = \frac{\cE(U)_{(t-s)-}}{\cE(U)_t} \left( V_t^x - \frac{\cE(U)_t}{\cE(U)_{(t-s)}}
			\int_{[t-s,t]}
			\left( \frac{\cE(U)_{u-}}{\cE(U)_{t-s}} \right)^{-1} \, \di \eta_u \right).
		\end{equation}
		From \eqref{Dolean} we obtain
		\begin{align*}
			\frac{\cE (U)_{(t-s)-}}{\cE(U)_t}  & = \re^{U_{(t-s)-} - U_t + \sigma_U^2 s/2}
			\prod_{t-s\leq u \leq t} \left( (1+\Delta U_u) \re^{-\Delta U_u}\right)^{-1} \\
			& = \re^{\widetilde{U}_s + \sigma_U^2 s/2} \prod_{0< u \leq s} \re^{-\Delta \widetilde{U}_u}
			(1- \Delta \widetilde{U}_u)^{-1},
		\end{align*}
		where we used that $\Delta \widetilde{U}_u = - \Delta U_{t-u}$ and $\Delta U_t=0$.
		On the other hand, an easy calculation using \eqref{Dolean} shows that the last expression is equal to $\cE(T)_s$, so that
		\begin{equation} \label{eq-flow-rev2}
			\frac{\cE (U)_{(t-s)-}}{\cE(U)_t} =\cE(T)_s.
		\end{equation}
		Since the right-hand sides of both \eqref{eq-flow-0} and \eqref{eq-flow-rev} define c\`adl\`ag processes in $s$, it is enough to show that for given $s$ they agree almost surely (where the exceptional null set may depend on~$s$). Since $\eta$ almost surely does not jump at the fixed time $t-s$, the integral over the compact interval $[t-s,s]$ in \eqref{eq-flow-rev} is  almost surely equal to the corresponding integral over the half-open interval $(t-s,s]$. Hence, by \eqref{eq-flow-rev} and \eqref{eq-flow-rev2}, Equation~\eqref{eq-flow-0} will follow if we can show that,
		almost surely for each fixed $s\in [0,t]$,
		\begin{equation} \label{eq-flow-rev4}
			- \frac{\cE(U)_t}{\cE(U)_{t-s}}
			\int_{(t-s,t]}
			\left( \frac{\cE(U)_{u-}}{\cE(U)_{t-s}} \right)^{-1} \, \di \eta_u = \int_{(0,s]} \cE(T)_{u-}^{-1} \diff \widetilde{L}_u.
		\end{equation}
		For this, define for fixed $s\in [0,t]$
		$$(U_u', \eta_u') := (U_{u+t-s}-U_{t-s}, \eta_{u+t-s}-\eta_{t-s}), \quad u \in [0,s].$$
		Then $\cE(U)_{u+t-s}/\cE(U)_{t-s} = \cE(U')_u$ by \eqref{Dolean} for $u\in [0,s]$ and
		\begin{align*}
			\frac{\cE(U)_t}{\cE(U)_{t-s}} \int_{(t-s,t]}
			\left( \frac{\cE(U)_{u-}}{\cE(U)_{t-s}} \right)^{-1} \, \diff \eta_u
			& = \cE(U')_s \int_{(0,s]} \cE(U')_{u-}^{-1} \diff \eta_u' .
		\end{align*}
		Denote by $$(\widehat{U_u'}, \widehat{\eta_u'}) := (U_{s-}',\eta_{s-}') - (U_{(s-u)-}', \eta_{(s-u)-}'), \quad u \in [0,s],$$
		the negative of the time-reversal of $(U',\eta')$ at $s$. Then, with $[\cdot,\cdot]$ denoting quadratic covariation,
		$$\cE(U')_s \int_{(0,s]} \cE(U')_{u-}^{-1} \di \eta_u' = \int_{(0,s]} \cE(\widehat{U'})_{u-} \diff \widehat{\eta_u'} + [\cE(\widehat{U'}),\widehat{\eta}]_s;$$
		this is Lemma~6.1 in \cite{LindnerMaller2005} when all jumps of $U$ are greater than $-1$, and for general $U$ it follows from the proof of Proposition~8.3 in \cite{BehmeLindner2012}
		(the left-hand side is the probability limit of the $-B^\sigma$ appearing in that proof, while the right-hand side is the probability limit of $A^\sigma$). Since $\cE(\widehat{U'})_r = 1 + \int_{(0,r]} \cE(\widehat{U'})_{u-} \diff \widehat{U_u'}$ we can rewrite the left-hand side of \eqref{eq-flow-rev4} as
		\begin{align*}
			- \frac{\cE(U)_t}{\cE(U)_{t-s}}
			\int_{(t-s,t]}
			\left( \frac{\cE(U)_{u-}}{\cE(U)_{t-s}} \right)^{-1} \, \di \eta_u
			& = - \int_{(0,s]} \cE(\widehat{U'})_{u-} \diff \left( \widehat{\eta_u'} + [\widehat{U'}, \widehat{\eta'}]_u\right) \\
			& = \int_{(0,s]} \cE(-\widetilde{U})_{u-} \diff \left( \widetilde{\eta}_u - [\widetilde{U}, \widetilde{\eta}]_u\right),
		\end{align*}
		where in the last line we used that
		$$(\widehat{U_u'}, \widehat{\eta_u'}) = (U_{s-}' - U_{(s-u)-}', \eta_{s-}' - \eta_{(s-u)-}') = (U_{t-} - U_{(t-u)-}, \eta_{t-} - \eta_{(t-u)-})  = -(\widetilde{U}_u, \widetilde{\eta}_u)$$
		for $u\in [0,s]$. But $\cE(-\widetilde{U})_{u} = \cE(T)_u^{-1}$ as a consequence of \eqref{Dolean},
		\begin{equation} \label{eq-eta-calc}
			\widetilde{\eta}_u  = \widetilde{L}_u + \sum_{t-u\leq r < t} \frac{\Delta U_r \Delta L_r}{1+\Delta U_r} + \sigma_{U,L} u
			 = \widetilde{L}_u + \sum_{0<r\leq u} \frac{\Delta \widetilde{U}_r \Delta \widetilde{L}_r}{1-\Delta \widetilde{U}_r} + \sigma_{U,L} u
		\end{equation}
		by \eqref{eq:etaviaUL} and \eqref{eq:deftimereverse}, and hence
		$$\widetilde{\eta}_u - [\widetilde{U}, \widetilde{\eta}]_u = \widetilde{\eta}_u - \sigma_{U,\eta} u - \sum_{0<r\leq u} \Delta \widetilde{U}_r \Delta \widetilde{\eta}_r = \widetilde{L}_u,$$
		thus establishing \eqref{eq-flow-rev4} and hence \eqref{eq-flow-0}.
		
		For the proof of \eqref{eq:time reversed GOU}, observe that by \eqref{eq:GOU-explicit}, the unique solution to the SDE~\eqref{eq:time reversed GOU} is given by $R_s^y = \cE(T)_s (y + \int_{(0,s]} \cE(T)_{u-}^{-1} \diff N_u)$, where by \eqref{eq:etaviaUL} the process $N$ is given by
		$$N_s = \widetilde{\eta}_s  - \sum_{0<r\leq s} \frac{\Delta T_r \Delta \widetilde{\eta}_r}{1+\Delta T_r} - s \sigma_{T,\widetilde{\eta}}, \quad s\in [0,t].$$
		Inserting \eqref{eq-eta-calc} for $\widetilde{\eta}$ and observing that $\Delta T_r = \Delta \widetilde{U}_r / (1-\Delta \widetilde{U}_r)$, $1+\Delta T_r = 1/(1-\Delta \widetilde{U}_r)$ and $\Delta \widetilde{\eta}_r = \Delta \widetilde{L}_r/(1-\Delta \widetilde{U}_r)$ we see that $N=\widetilde{L}$. This shows that $(R_s^y)_{s\in [0,t]}$ is the unique solution of~\eqref{eq:time reversed GOU}.
	\end{proof}
	
	From Theorems~\ref{thm:dualGOU} and~\ref{thm:time_reversed_GOU} we deduce that, if all jumps of $U$ are greater than $-1$, then the process $(R_s^y)_{s\in [0,t]}$, $y\in \bR$, given in \eqref{eq:time reversed GOU} is Siegmund dual to $(V_s^x)_{s\in [0,t]}$, $x\in \bR$, with the obvious notion of Siegmund duality for processes defined only on $[0,t]$.
	
	\begin{corollary} \label{cor-backwards}
		Let $(V_s^x)_{s\geq 0}$, $x\in \bR$, be the  solution of the SDE \eqref{eq:GOUSDE} for the bivariate L\'evy process $(U_s,L_s)_{s\geq 0}$ with $\Delta U_s>-1$ for all $s\geq 0$. Fix $t>0$ and consider the stochastic process $(R_s^y)_{s\in [0,t]}$, $y\in \bR$, corresponding to the inverse stochastic flow defined in \eqref{eq:time reversed GOU}. Then $(R_s^y)_{s\in [0,t]}$, $y\in \bR$,  is Siegmund dual to $(V_s^x)_{s\in [0,t]}$, $x\in \bR$.
	\end{corollary}

	\begin{proof}
		The processes $(T,\widetilde{\eta})$ appearing in Theorem~\ref{thm:time_reversed_GOU}, and $(W,K)$ appearing in Theorem~\ref{thm:dualGOU} are given  by (compare also \eqref{eq-eta-calc} and \eqref{eq:etaviaUL})
		\begin{align*}
		T_s & = \widetilde{U}_s + \sigma_U^2 s + \sum_{0<r\le s} \frac{(\Delta \widetilde{U}_r)^2}{1-\Delta \widetilde{U}_r}, \\
		\widetilde{\eta}_s & = \widetilde{L}_s + \sigma_{U,L} s + \sum_{0<r\leq s} \frac{\Delta \widetilde{U}_r \Delta \widetilde{L}_r}{1-\Delta \widetilde{U}_r}, \\
		W_s & = - U_s + \sigma_U^2 s + \sum_{0<r\leq s} \frac{(\Delta U_r)^2}{1+\Delta U_r}, \quad \mbox{and}\\
		K_s & = - \eta_s = -L_s + \sigma_{U,L} s + \sum_{0<r\leq s} \frac{\Delta U_r \Delta L_r}{1+\Delta U_r},
		\end{align*}
		for $s\in [0,t]$. Since  the time-reversed L\'evy process $(\tilde{U}_s,\tilde{L}_s)_{s\in [0,t]}$ is equal in law to $(-U_s, -L_s)_{s\in [0,t]}$ we conclude that $(T_s,\widetilde{\eta}_s)_{s\in [0,t]}$ and $(W_s,K_s)_{s\in [0,t]}$ are equal in law. The claim follows from Theorems~\ref{thm:dualGOU} and~\ref{thm:time_reversed_GOU}.
	\end{proof}
	
	\begin{remark} 
		Corollary~\ref{cor-backwards} is not surprising, since reversibility of the stochastic flow is intrinsically related to duality, as discussed in various articles, see e.g.~\cite{Sigman2000}. Also in our case, it is possible to deduce Theorem~\ref{thm:dualGOU} directly from Theorem~\ref{thm:time_reversed_GOU}. To see this, let $(U,L)$, $V^x_s$ and $R^y_s$ be as in
		Corollary~\ref{cor-backwards}. For $u \in [0,t]$ denote by $\varphi_{u,t} : \bR \to \bR$ the stochastic flow which transports $x=V_u(\omega)$ to $V_t(\omega)$, i.e. $\varphi_{u,t}(x) = V_t(\omega)$ conditional on $V_u=x$ (we suppress the superscript~$x$ here and prefer to work with $V_t|V_0=x$ rather than $V_t^x$ for the moment). By time-homogeneity of $V$, we have for $s\in [0,t]$
		$$\PP(V_s \geq y | V_0 = x) = \PP(V_t \geq y | V_{t-s} = x) =
		\PP(\varphi_{t-s,t}(x) \geq y) = \PP(\varphi_{t-s,t}^{-1} (y) \leq x),$$
		since $\varphi_{t-s,t}$ is strictly increasing and bijective (the exact form of the flow can be read off from~\eqref{eq-flow-expl}). But $\varphi_{t-s,t}^{-1}(y) = R_{s-}^y$ by \eqref{eq-flow-0} and \eqref{eq-flow-1}, hence
		$$\PP(V_s\geq y|V_0=x) = \PP(R_{s-}^y \leq x) = \PP(R_s^y\leq x) = \PP(R_s\leq x |R_0 = y),$$ showing that $(R_s^y)_{s\in [0,t]}$, $y\in \bR$, is dual to $(V_s^x)_{s\in [0,t]}$, $x\in \bR$. By the proof of Corollary~\ref{cor-backwards}, $(T_s,\widetilde{\eta}_s)_{s\in [0,t]}$ is equal in law to $(W_s,K_s)_{s\in [0,t]}$, showing that the process given in Theorem~\ref{thm:dualGOU} (call it $\overline{R}_s^y$ for the moment) is also dual to $R$. Hence we have given another proof of the duality of $(\overline{R}_s^y)$ to $(V_s^x)$ in Theorem~\ref{thm:dualGOU}. This last proof is in line with the reasoning given in Sigman and Ryan~\cite[Cor.~3.1 (1)]{Sigman2000}, who consider $[0,\infty)$-valued time-homogeneous Markov processes with certain properties.
	\end{remark}

	\begin{remark}
		According to \cite[Prop. 4.4]{Jansen2014} the Siegmund dual GOU process derived in Theorem~\ref{thm:dualGOU} is even \emph{strongly pathwise dual}, i.e. both $V$ and $R$ can be realised on a suitable probability space such that
		$$\{ V_t^x\geq y \} = \{ R_t^y\leq x \}\quad \text{almost surely}.$$
		This follows from the construction method described in \cite{CliffordSudbury1985}, but it is not a direct consequence of our proof, as in \eqref{eq:BLMLemma3.1} we use a distributional inequality. Still, using time-reversals, one could replace \eqref{eq:BLMLemma3.1} by an almost sure statement (see \cite[Lemma 6.1]{LindnerMaller2005} for a special case), and the resulting dual in this case turns out to coincide with the inverse stochastic flow obtained in Theorem \ref{thm:time_reversed_GOU}.
	\end{remark}	

	\section*{Acknowledgements}
	The authors thank two anonymous referees whose comments and suggestions significantly helped to improve the paper's overall presentation and led to additional results such as Corollary \ref{c-subordinator-ruin}.

\bibliographystyle{plain} 
\bibliography{bib_siegmund}

\end{document}